\documentclass[11pt]{amsart}
\usepackage[T1]{fontenc}
\usepackage{accanthis}
\usepackage[margin=1.35in]{geometry}
\usepackage{amscd,amsmath,amsxtra,amsthm,amssymb,stmaryrd,xr,mathrsfs,mathtools,enumerate,commath, comment}
\usepackage{stmaryrd}
\usepackage{xcolor}
\usepackage{commath}
\usepackage{comment}
\usepackage{tikz-cd}
\usepackage{longtable} % for 'longtable' environment
\usepackage{pdflscape} % for 'landscape' environment
\usepackage{booktabs}
\usepackage{hyperref}
\definecolor{vegasgold}{rgb}{0.77, 0.7, 0.35}
\definecolor{darkgoldenrod}{rgb}{0.72, 0.53, 0.04}
\definecolor{gold(metallic)}{rgb}{0.83, 0.69, 0.22}
\hypersetup{
 colorlinks=true,
 linkcolor=darkgoldenrod,
 filecolor=brown,      
 urlcolor=gold(metallic),
 citecolor=darkgoldenrod,
 pdftitle={Remarks on Greenberg's conjecture for Galois representations associated to elliptic curves},
 }

\usepackage[all,cmtip]{xy}

\DeclareFontFamily{U}{wncy}{}
\DeclareFontShape{U}{wncy}{m}{n}{<->wncyr10}{}
\DeclareSymbolFont{mcy}{U}{wncy}{m}{n}
\DeclareMathSymbol{\Sh}{\mathord}{mcy}{"58}
\usepackage[T2A,T1]{fontenc}
\usepackage[OT2,T1]{fontenc}

\newtheorem{theorem}{Theorem}[section]
\newtheorem{lemma}[theorem]{Lemma}
\newtheorem{ass}[theorem]{Assumption}
\newtheorem*{theorem*}{Theorem}
\newtheorem*{ass*}{Assumption}
\newtheorem{lthm}{Theorem}
\newtheorem{definition}[theorem]{Definition}
\newtheorem{corollary}[theorem]{Corollary}

\newtheorem{conjecture}[theorem]{Conjecture}
\newtheorem{proposition}[theorem]{Proposition}

\newcommand{\cF}{\mathcal{F}}

\newcommand{\cK}{\mathcal{K}}

\newcommand{\cG}{\mathcal{G}}

\newcommand{\ord}{\mathrm{ord}}

\newcommand{\cH}{\mathcal{H}}

\newcommand{\Z}{\mathbb{Z}}

\newcommand{\Q}{\mathbb{Q}}
\newcommand{\F}{\mathbb{F}}

\newcommand{\cO}{\mathcal{O}}

\newcommand{\Hom}{\mathrm{Hom}}

\newcommand{\coker}{\mathrm{coker}}

\newcommand{\op}[1]{\operatorname{#1}}

\newcommand\mtx[4] { \left( {\begin{array}{cc}
 #1 & #2 \\
 #3 & #4 \\
 \end{array} } \right)}

 \newcommand{\widebar}[1]{\mkern 2.5mu\overline{\mkern-2.5mu#1\mkern-2.5mu}\mkern 2.5mu}

\numberwithin{equation}{section}

\begin{document}

\title[On Greenberg's conjecture]{Remarks on Greenberg's conjecture for Galois representations associated to elliptic curves}

\author[A.~Ray]{Anwesh Ray}
\address[Ray]{Centre de recherches mathématiques,
Université de Montréal,
Pavillon André-Aisenstadt,
2920 Chemin de la tour,
Montréal (Québec) H3T 1J4, Canada}
\email{ar2222@cornell.edu}


\maketitle

\begin{abstract}
Let \( E \) be an elliptic curve defined over \( \mathbb{Q} \), and let \( p \) be an odd prime at which \( E \) has good ordinary reduction. Consider the \( p \)-primary Selmer group \( \operatorname{Sel}_{p^\infty}(\mathbb{Q}_\infty, E) \) over the cyclotomic \( \mathbb{Z}_p \)-extension of \( \mathbb{Q} \), and denote its (algebraic) \(\mu\)-invariant by \( \mu_p(E) \). Let \( \bar{\rho}_{E, p} : \operatorname{Gal}(\bar{\mathbb{Q}}/\mathbb{Q}) \to \operatorname{GL}_2(\mathbb{Z}/p\mathbb{Z}) \) be the Galois representation arising from the action of the absolute Galois group on the \( p \)-torsion of \( E \). Greenberg conjectured that if \( \bar{\rho}_{E, p} \) is reducible, then there exists a rational isogeny \( E \to E' \) of \( p \)-power degree such that \( \mu_p(E') = 0 \). In this article, we investigate this conjecture by establishing sufficient Galois-theoretic criteria for its validity, formulated in terms of the representation \( \bar{\rho}_{E, p} \). Our approach relies on a fundamental result of Coates and Sujatha concerning the fine Selmer group. Furthermore, in the case where \( \bar{\rho}_{E, p} \) is irreducible, we show that our conditions imply \( \mu_p(E) = 0 \) under the additional assumption that the classical Iwasawa \(\mu\)-invariant vanishes for the splitting field \( \mathbb{Q}(E[p]) = \bar{\mathbb{Q}}^{\ker \bar{\rho}_{E, p}} \).
\end{abstract}

\section{Introduction}
Iwasawa studied growth questions of class groups in certain infinite towers of number fields. Given a number field $F$ and a prime number $p$, let $F_{\infty}$ denote the cyclotomic $\Z_p$-extension of $F$. Letting $F_n\subset F_{\op{cyc}}$ be such that $[F_n:F]=p^n$, denote by $\op{Cl}_p(F_n)$ the $p$-Sylow subgroup of the class group of $F_n$. Iwasawa showed that there exists $n_0\in \Z_{\geq 0}$ such that for all $n\geq n_0$,
\begin{equation}\label{first equation} |\op{Cl}_p(F_n)|=p^{\left(p^n\mu_p(F)+\lambda_p(F) n+\nu_p(F)\right)},\end{equation} where $\mu_p(F), \lambda_p(F)\in \Z_{\geq 0}$ and $\nu_p(F)\in \Z$. We shall refer to $\mu_p(F)$ as the \emph{classical $\mu$-invariant} of $F$. Iwasawa conjectured that $\mu_p(F)=0$ for all prime numbers $p$ and all number fields $F$. This conjecture was proved by Ferrero and Washington \cite{ferrero1979iwasawa} for all abelian number fields. 
\par The Iwasawa theory of Galois representations arising from motives leads to very deep questions. Let $p$ be an odd prime number and $E_{/\Q}$ be an elliptic curve. Mazur \cite{mazur1972rational} initiated the Iwasawa theory of elliptic curves, and studied the algebraic structure of the Selmer group over the cyclotomic $\Z_p$-extension of $\Q$. The (algebraic) Iwasawa $\mu$-invariant $\mu_p(E)$ is defined in terms of the algebraic structure of this Selmer group, when viewed as a module over the Iwasawa algebra $\Lambda$. When $E$ has good ordinary reduction at $p$, these Selmer groups are known to be cotorsion as $\Lambda$-modules, thanks to the work of Kato \cite{kato2004p}. This property is especially favorable when it comes to studying the properties of the $\mu$-invariant. When the $\mu$-invariant vanishes, the Selmer group is cofinitely generated as a $\Z_p$-module. 

\par We recall a conjecture due to Greenberg \cite[Conjecture 1.11]{greenbergmain} on the vanishing of the $\mu$-invariant. Let $E[p]\subset E(\bar{\Q})$ be the $p$-torsion subgroup, we note that as an abstract abelian group, $E[p]$ is isomorphic to $(\Z/p\Z)^2$. Let \[\bar{\rho}_{E, p}:\op{Gal}(\bar{\Q}/\Q)\rightarrow \op{GL}_2(\Z/p\Z)\] be the representation of the absolute Galois group on $E[p]$. The following conjecture is of paramount importance and has influenced major trends in Iwasawa theory. 

\begin{conjecture}[Greenberg]\label{greenberg conj}
    Let $E_{/\Q}$ be an elliptic curve and $p$ a prime number at which $E$ has good ordinary reduction.
    \begin{enumerate}
        \item\label{p1 of Greenberg} Suppose that the Galois representation $\bar{\rho}_{E,p}$ is reducible. Then, $E$ is $\Q$-isogenous to an elliptic curve $E'$ such that $\mu_p(E')=0$. 
        \item Suppose that the Galois representation $\bar{\rho}_{E,p}$ is irreducible, then, $\mu_p(E)=0$. 
    \end{enumerate}
\end{conjecture}
There are examples of elliptic curves $E_{/\Q}$ for which $\bar{\rho}_{E,p}$ is reducible and for which $\mu_p(E)>0$. The first such example is due to Mazur, see \cite[section 9]{mazur1972rational}.
\subsection{Main results}
\par In this paper, we study Greenberg's conjecture from a new perspective. Before discussing the main results, let us introduce some further notation. Let $L_\infty$ be the cyclotomic $\Z_p$-extension of $L:=\Q(E[p])$ and set $G:=\op{Gal}(L_\infty/\Q_{\infty})$. Choose an embedding $\iota_p:\bar{\Q}\hookrightarrow \bar{\Q}_p$, and let $\tilde{\beta}$ be the prime of $\bar{\Q}$ that lies above $p$, prescribed by this embedding. Let $\beta$ be the prime of $L_\infty$ that lies below $\tilde{\beta}$, and $I_{\beta}$ be the inertia group at $\beta$. Let $\Sigma$ be a finite set of prime numbers that contains $p$ and the primes of bad reduction of $E$, and let $\Q_{\Sigma}$ be the maximal algebraic extension of $\Q$ unramified outside $\Sigma$. Since $E$ is assumed to be ordinary at $p$, there is a unique $\op{G}_p$-invariant line $\bar{C}$ contained in $E[p]$, such that the inertia group at $p$ acts on $\bar{C}$ via the mod-$p$ cyclotomic character. Note that the Galois group $\op{Gal}(\Q_\Sigma/L_\infty)$ acts trivially on $E[p]$. At a prime $\eta$ of $L_\infty$ that lies above a prime in $\Sigma$, let $f_{\eta}$ be the restriction of $f$ the decomposition group at $\eta$. The residual Selmer group over $L_\infty$ is defined as follows
\[\begin{split}\op{Sel}(L_\infty, E[p]):=&\{f\in \Hom\left(\op{Gal}(\Q_\Sigma/L_\infty), E[p]\right)\mid \\ & f_\eta=0\text{ for all primes }\eta\nmid p \text{ that lie above a prime of }\Sigma, \\ & \text{ and }f(I_{\beta})\subseteq \bar{C}\}.\end{split}\]
A homomorphism $f\in \Hom\left(\op{Gal}(\Q_\Sigma/L_\infty), E[p]\right)$ is $G$-equivariant if for any $g\in G$, 
\[f(\tilde{g} x\tilde{g}^{-1})=g f(x),\] where $\tilde{g}$ is a lift of $g$ to $\op{Gal}(\Q_\Sigma/\Q_\infty)$. Let $\op{Sel}(L_\infty, E[p])^G$ be the subgroup of $\op{Sel}(L_\infty, E[p])$ consisting Selmer homomorphisms that are $G$-equivariant. 
\begin{conjecture}\label{main conjecture}
        Suppose that $\bar{C}$ is not a $G$-stable submodule of $E[p]$. Then, the image of $\op{Sel}(L_\infty, E[p])^G$ in $\Hom(I_\beta, \bar{C})$ is finite.
\end{conjecture}

Thus, if $\bar{C}$ is not a $G$-stable submodule of $E[p]$, then this poses a restriction on the ramification of $\bar{C}$ for $G$-equivariant Selmer classes. When $\bar{C}$ is a $G$-submodule of $E[p]$, it can lead to the $\mu$ invariant being non-zero. All the examples where the $\mu$-invariant is positive in the literature involve situations where $\bar{C}$ is a $G$-submodule of $E[p]$. This leads provides some further evidence for the conjecture. The above is a Galois theoretic criterion expressed purely in terms of the residual representation, which as we shall see gives a different formulation of Greenberg's conjectures. %It follows from our arguments that the above condition is in fact in many situations equivalent to Greenberg's conjecture.
In establishing our result, we crucially leverage results of Coates and Sujatha on the vanishing of the $\mu$-invariant of the \emph{fine Selmer group}. We prove the first part of Conjecture \ref{greenberg conj}, provided Conjecture \ref{main conjecture} holds for the $\Q$-isogeny class of $E$.
\begin{lthm}\label{main thm}
    Let $E$ be an elliptic curve over $\Q$ with good ordinary reduction at $p$ such that $E[p]$ is reducible. Assume that the Conjecture \ref{main conjecture} is true for all elliptic curves over $\Q$ that are isogenous to $E$. Then, there exists an elliptic curve $E'$ over $\Q$ that is isogenous to $E$, such that the $\mu$-invariant of $\op{Sel}_{p^\infty}(\Q_\infty, E')$ is $0$. Moreover, the isogeny $E\rightarrow E'$ has degree $p^{\mu_p(E)}$.
\end{lthm}

Thus we have provided an explicit criterion for the validity of Greenberg's conjecture. Moreover, we show that the property that $\mu_p(E)$ vanishes can be detected precisely from the structure of the representation $\bar{\rho}_{E,p}$, see Theorem \ref{mu equals zero condition}. We prove that the second part of this conjecture follows from certain additional conditions.
\begin{lthm}\label{main thm 2}
    Let $E_{/\Q}$ be an elliptic curve with good ordinary reduction at an odd prime $p$ and assume that
    \begin{enumerate}
        \item the residual representation $\bar{\rho}_{E, p}$ is irreducible.
        \item The Conjecture \ref{main conjecture} holds for $E$.
         \item The classical Iwasawa $\mu$-invariant $\mu_p(L)$ vanishes.
    \end{enumerate}
    Then, the $\mu$-invariant of the Greenberg Selmer group $\op{Sel}_{p^\infty}(\Q_\infty, E)$ is $0$.
\end{lthm}

The main novelty in the above results, is that the algebraic structure of the fine Selmer group plays an important role in establishing them. Indeed, the Conjecture \ref{main conjecture} gives a precise criterion for the fine Selmer group and the Greenberg Selmer group to have the same $\mu$-invariant. We explain this in greater detail below.

\subsection{Method of proof} The results are proven by proving an explicit relationship between the Selmer group over $\Q_\infty$ and the fine Selmer group. This latter Selmer group is defined by imposing strict conditions at the prime above $p$. The algebraic properties of the fine Selmer group closely resemble those of class groups, and the $\mu$-invariants of these Selmer groups are always conjectured to vanish. Suppose that the Conjecture \ref{main conjecture} holds for all elliptic curve that are isogenous to $E$. Then, we show that if either $\bar{\rho}_{E, p}$ is irreducible, or if $\bar{\rho}_{E,p}$ is reducible and satisfies some further conditions, then the $\mu$-invariant of the Greenberg Selmer group vanishes if and only if the $\mu$-invariant of the fine Selmer group vanishes. We show that any elliptic curve $E$ for which $\bar{\rho}_{E, p}$ is reducible is isogenous to another elliptic curve $E'$ for which these additional conditions on $\bar{\rho}_{E',p}$ are satisfied. The results of Coates-Sujatha \cite{coates2005fine} and Ferrero-Washington \cite{ferrero1979iwasawa} 
 together imply that the $\mu$-invariant of the fine Selmer group vanishes when the residual representation is reducible (cf. Theorem \ref{CS thm}). We leverage the vanishing of the $\mu$-invariant of the fine Selmer group to deduce that the $\mu$-invariant of $E'$ vanishes, thus proving Theorem \ref{main thm}. The Theorem \ref{main thm 2} is proven via a similar technique. It is shown that if $\bar{\rho}_{E,p}$ is irreducible, then Greenberg's conjecture is equivalent to a conjecture of Coates and Sujatha on the vanishing of the $\mu$-invariant of the fine Selmer group (cf. Corollary \ref{cor 5.5}). The results establishing the relationship between the Greenberg Selmer group and the fine Selmer group hold in a more general context, namely, for ordinary Galois representations (cf. Theorem \ref{main 1}). 

 \subsection{Outlook} The methods developed in this paper should lead to interesting generalizations of Greenberg's conjecture for ordinary Galois representations associated to modular forms and abelian varieties. Such questions have not been pursued in this paper, however would certainly be of interest to study in the future. It is of natural interest to ascertain if the statement Theorem \ref{main thm 2} generalizes in some sense to the case of Kobayashi's signed Selmer groups \cite{kobayashi2003iwasawa} associated to elliptic curves with supersingular reduction at $p$. The main difficulty here is that Kobayashi's Selmer groups have not been defined over fields other that the rational numbers in which $p$ does not split completely. We work with a suitably defined primitive residual Selmer group considered over the cyclotomic $\Z_p$-extension of $L$.

\subsection{Related work} Let us discuss related work towards Greenberg's conjecture in the residually reducible case. Greenberg and Vatsal \cite{greenberg2000iwasawa} showed that if $\bar{\rho}_{E,p}$ is reducible and $E[p]$ contains a $1$-dimensional line which is $\op{G}_{\Q}$-stable on which the action is via a character which is either 
\begin{enumerate}
    \item unramified at $p$ and odd,
    \item ramified at $p$ and even,
\end{enumerate}
then, $\mu_p(E)=0$. The case that remains is when the Galois representation $\bar{\rho}_{E,p}$ is indecomposable of the form $\mtx{\varphi_1}{0}{\ast}{\varphi_2}$, where $\varphi_2$ is unramified at $p$ and even. In this case, Greenberg conjectures that $\mu_p(E)=0$. The conjecture has been studied for $p=3$ by Hachimori in the case when $E(\Q)[3]\neq 0$ (i.e., $\varphi_2=1$), cf. \cite[Theorem 1.1]{hachimori2004mu}. On the other hand, Trifkovic \cite{trifkovic2005vanishing} constructs infinitely many examples for $p=3,5$ of elliptic curves $E_{/\Q}$ for which $\bar{\rho}_{E,p}$ is of the form $\mtx{\varphi_1}{0}{\ast}{\varphi_2}$ described above, and such that $\mu_p(E)=0$. The case of interest falls under "type 3" under the classification in section \ref{s 6}. 
\par We note here that the methods developed in this paper do not immediately generalize to elliptic curves over arbitrary number fields, since the arguments make use of the fact that there is only one prime of $\Q_\infty$ that lies above $p$. For a number field in which $p$ splits into $2$ or more primes, the same reasoning no longer applies. In fact, Drinen \cite{drinen2003iwasawa} proved that there are large enough number fields over which the analogue of part \eqref{p1 of Greenberg} of Greenberg's conjecture does not hold.

\subsection{Organization} Including the introduction, the article consists of 6 sections. In section \ref{s 2}, we introduce basic notation and conventions, introduce the Greenberg Selmer group associated with an ordinary Galois representations. These Selmer groups are modules over the Iwasawa algebra, the structure theory of such modules leads to the definition of the $\mu$-invariant. When the Selmer group is cotorsion over the Iwasawa algebra, the $\mu$-invariant vanishes precisely when it is cofinitely generated as a $\Z_p$-module. We conclude section \ref{s 2} by discussing the relationship between the Bloch-Kato Selmer group and the Greenberg Selmer group. It turns out that for elliptic curves $E_{/\Q}$, the Greenberg Selmer group coincides with the classical Selmer group for which the local conditions are defined via Kummer maps. The Greenberg Selmer groups are however more convenient to work with when employing Galois cohomological arguments. 
\par In section \ref{s 3}, we introduce a Selmer group associated to the residual representation. Such Selmer groups were considered by Greenberg and Vatsal in \cite{greenberg2000iwasawa} in studying the role of congruences between elliptic curves in Iwasawa theory. In \emph{loc. cit.}, certain imprimitive residual Selmer groups are defined, i.e., the local conditions at a number of primes are not imposed on these Selmer groups. It is necessary to work with such imprimitive conditions when studing the effect of congruences on the $\lambda$-invariant. In this article however, we only study the $\mu$-invariant, and therefore work with primitive residual Selmer groups. The section ends with Proposition \ref{prop 3.3}, which shows that the $\mu$-invariant vanishes for the Greenberg Selmer group of an ordinary representation if and only if the (primitive) residual Selmer group is finite. Such a result is well known to the experts, however, we do include it for completeness.
\par The section \ref{s 4} is devoted to the definition and basic properties of the fine Selmer group associated to a continuous Galois representation. We discuss a conjecture of Coates and Sujatha on the vanishing of the $\mu$-invariant. We end section \ref{s 4} by recalling a key result on the vanishing of this $\mu$-invariant.
\par In section \ref{s 5} we introduce a key assumption on the residual representation. This assumption allows us to relate the Greenberg Selmer group to the fine Selmer group. It is shown that the residual fine Selmer group is of finite index in the residual Greenberg Selmer group, provided the residual representation satisfies the additional hypothesis. It the end of this section \ref{s 5}, we give a proof of Theorem \ref{main thm 2}.
\par Finally, section \ref{s 6} is devoted to the proof of Theorem \ref{main thm}. We give a classification of the residual representation into three types. It is shown that the $\mu$-invariant is positive when the residual representation is of type 1, and is zero when it is of type 2 or 3. It is in the case when the representation is of type 3 that the results from section \ref{s 5} are applied. We then recall a theorem of Schneider, which describes the difference between the $\mu$-invariants of isogenous elliptic curves. We use this result, along with our classification theorem to prove Theorem \ref{main thm}. 

\subsection*{Acknowledgments} The author's research is supported by the CRM Simons postdoctoral fellowship. He thanks Jeffrey Hatey, Antonio Lei and Shaunak V. Deo for some helpful comments. He also thanks the referee for helpful suggestions and corrections that have benefited the article.

\subsection*{Data availability} No data was analyzed in this work.

\subsection*{Conflict of interest} There is no conflict of interest.

\section{Notation and Preliminaries}\label{s 2}
\subsection{Notation}
In this section, we introduce some standard notation.
\begin{itemize}
    \item Throughout $p$ will denote an odd prime number and $K/\Q_p$ a finite extension.
    \item Let $\F_p$ be the field with $p$ elements.
    \item Denote by $\cO$ the valuation ring in $K$. Let $\varpi$ be a uniformizer of $\cO$ and set $\F:=\cO/\varpi$. 
    \item Let $\bar{\Q}$ be an algebraic closure of $\Q$. For a subfield $F$ of $\bar{\Q}$, we set $\op{G}_F$ to denote the absolute Galois group $\op{Gal}(\bar{\Q}/F)$. We make the convention that all algebraic extensions of $\Q$ considered are contained in $\bar{\Q}$.
    \item For each prime $\ell$, set $\op{G}_\ell:=\op{Gal}(\bar{\Q}_\ell/\Q_\ell)$. Choose an embedding $\iota_\ell: \bar{\Q}\hookrightarrow \bar{\Q}_\ell$; set $\iota_\ell^*: \op{G}_\ell\hookrightarrow \op{G}_\Q$ to denote the induced inclusion.
    \item Given a finite set of prime numbers $\Sigma$, set $\Q_\Sigma$ to be the maximal algebraic extension of $\Q$ in which all primes $\ell\notin \Sigma$ are unramified.
    \item Let $\cF$ be an extension of $\Q$ contained in $\Q_\Sigma$, set $H^i(\Q_\Sigma/\cF, \cdot):=H^i\left(\op{Gal}(\Q_\Sigma/\cF), \cdot\right)$.
    \item Let $\mu_{p^n}\subset \bar{\Q}$ denote the $p^n$-th roots of unity, and $\Q(\mu_{p^n})$ be the cyclotomic field generated by $\mu_{p^n}$. Set $\Q(\mu_{p^\infty})$ to denote the union of number fields $\Q(\mu_{p^n})$ for $n\in \Z_{\geq 1}$. 
    \item Set $\mathcal{G}_\infty:=\op{Gal}(\Q(\mu_{p^\infty})/\Q)$ and $\chi: \mathcal{G}_\infty\xrightarrow{\sim}\Z_p^\times$ be the $p$-adic cyclotomic character. 
    \item Denote by $\Q_\infty$ the unique $\Z_p$-extension of $\Q$ that is contained in $\Q(\mu_{p^\infty})$, this is the \emph{cyclotomic $\Z_p$-extension of $\Q$}. 
    \item Set $\Gamma:=\op{Gal}(\Q_\infty/\Q)$ and $\Delta:=\op{Gal}(\Q(\mu_p)/\Q)$, note that $\mathcal{G}_\infty\simeq \Delta\times \Gamma$. 
    \item For $n\in \Z_{\geq 1}$, let $\Q_n$ be the subfield of $\Q_\infty$ with $\op{Gal}(\Q_\infty/\Q_n)= \Gamma^{p^n}$. Identify $\op{Gal}(\Q_n/\Q)$ with $\Gamma_n:=\Gamma/\Gamma^{p^n}$.
        \item Given a prime $\eta$ of $\Q_\infty$, denote by $\Q_{\infty, \eta}$ the union of completions of the finite layers $\Q_n$ at $\eta$.
        \item Let $\eta_p$ be the unique prime of $\Q_\infty$ that lies above $p$.
    \item The \emph{Iwasawa algebra} $\Lambda$ over $\cO$ is the completed group algebra $\varprojlim_n \cO[\Gamma_n]$. 
    \item Let $\gamma$ be a topological generator of $\Gamma$, setting $T:=(\gamma-1)$ we identify $\Lambda$ with the formal power series ring $\cO\llbracket T\rrbracket$.  
\end{itemize}

\subsection{The Greenberg Selmer group}

\par We recall the definition of the Greenberg Selmer group associated to an ordinary Galois representation. These Selmer groups are considered over $\Q_\infty$, and were introduced in \cite{greenbergmain}. We follow the notation and conventions in \cite{greenberg2000iwasawa}. Let $n\geq 2$ be an integer, and $M$ be a free $\cO$-module of rank $n$, equipped with a continuous $\cO$-linear action of $\op{G}_\Q$. Choose an $\cO$-basis for $M$, and identify the group of $\cO$-linear automorphisms of $M$ with $\op{GL}_n(\cO)$. The Galois action on $M$ is encoded by a continuous Galois representation 
\[\rho_{M}:\op{G}_\Q\rightarrow \op{GL}_n(\cO). \] Tensoring with $K$, we obtain the representation on the $K$-vector space $M_K:=M\otimes_{\cO} K$, denoted
\[\rho_{M_K}: \op{G}_\Q\rightarrow \op{GL}_n(K). \] 

Set $d:=\dim_K M_K$, and $d^\pm$ to denote the $K$-dimensions of the plus and minus eigenspaces of $M_K$ for the action of complex conjugation. We set $A:=M_K/M$; we take note of the fact that $A\simeq (K/\cO)^d$. 

\begin{ass}\label{basic ass}
    With respect to notation above, assume that 
    \begin{enumerate}
        \item there are finitely many prime numbers at which $\rho_M$ is ramified.
        \item The representation $\rho_{M_K}$ is irreducible. 
        \item There exists a $\op{G}_p$-stable $K$-subspace $W$ of $M_K$ of dimension $d^+$ such that the action of $\op{G}_p$ on $M_K/W$ is unramified.
    \end{enumerate}
\end{ass}

A Galois representation satisfying the conditions of \cite[p.98]{greenbergmain} is referred to as \textit{ordinary} and satisfies the above assumption. Let $C$ to be the image of $W$ in $A$, and note that $C\simeq (K/\cO)^{d^+}$. Setting $D:=A/C$, we find that $D\simeq (K/\cO)^{d^-}$. Following \cite[section 2]{greenberg2000iwasawa}, we recall the definition of the Greenberg Selmer group associated with the pair $(A, C)$. For $\ell\neq p$, set 
\[\cH_\ell(\Q_\infty, A):=\prod_{\eta|\ell} H^1(\Q_{\infty,\eta}, A), \] where $\eta$ runs through the primes of $\Q_\infty$ that lie above $\ell$. The set of primes $\eta$ of $\Q_\infty$ that lie above any given rational prime is finite. Denote by $I_{\eta_p}$ the inertia group of $\op{Gal}\left(\widebar{\Q_{\infty, \eta_p}}/\Q_{\infty, \eta_p}\right)$. We let $L_{\eta_p}$ be defined as follows
\[L_{\eta_p}:=\op{ker}\left(H^1(\Q_{\infty, \eta_p}, A)\xrightarrow{\kappa_p} H^1(I_{\eta_p}, D)\right),\]where $\kappa_p$ is the kernel of the composite of the natural maps
\[\begin{split} & H^1(\Q_{\infty, \eta_p}, A)\rightarrow H^1(I_{\eta_p}, A)\\
& H^1(I_{\eta_p}, A)\rightarrow H^1(I_{\eta_p}, D).\\ \end{split}\]
The first of the above maps is the restriction map and the second map is induced by the map $A\rightarrow D$. The local condition at $p$ is prescribed as follows $\cH_p(\Q_\infty, A):=H^1(\Q_\infty, A)/L_{\eta_p}$.
\par The Selmer group over $\Q_\infty$ is defined as follows
\[S_A(\Q_\infty):=\op{ker}\left(H^1(\Q_\Sigma/\Q_\infty, A)\rightarrow \bigoplus_{\ell\in \Sigma} \cH_\ell(\Q_\infty, A)\right)\]
where $\Sigma$ is a finite set of prime numbers containing $p$ and all prime numbers at which $\rho_M$ is ramified. As is well known, the Selmer group defined above is independent of the choice of primes $\Sigma$. Given a module $\mathbf{M}$ over $\Lambda$, set $\mathbf{M}^\vee:=\Hom_{\Z_p}\left(\mathbf{M}, \Q_p/\Z_p\right)$ to denote its Pontryagin dual. We say that $\mathbf{M}$ is cofinitely generated (resp. cotorsion) over $\Lambda$ if $\mathbf{M}^\vee$ is finitely generated (resp. torsion) as a $\Lambda$-module.

\par The Selmer group $S_A(\Q_\infty)$ is cofinitely generated as a module over $\Lambda$. Throughout we make the following assumption.
\begin{ass}[Cotorsion hypothesis]\label{cotorsion hyp}
    The Selmer group $S_A(\Q_\infty)$ is cofinitely generated over $\Lambda$.
\end{ass}
This assumption is known to hold in the main case of interest, namely for ordinary Galois representations associated with rational elliptic curves, cf. \cite[Theorem 1.5]{greenbergmain}.
\subsection{The Iwasawa $\mu$-invariant}
\par Let $\mathbf{M}$ be a cofinitely generated and cotorsion $\Lambda$-module. We recall the definition of the Iwasawa $\mu$-invariant associated to $\mathbf{M}$. 
\par A map of $\Lambda$-modules $\mathbf{M}_1\rightarrow \mathbf{M}_2$ is said to be a \emph{pseudo-isomorphism} if its kernel and cokernel are both finite. A polynomial $f(T)\in \Lambda$ is \emph{distinguished} if it is a monic polynomial and all non-leading coefficients are divisible by $\varpi$. According to the structure theorem of $\Lambda$-modules \cite[Chapter 13]{washington1997introduction}, there is a pseudoisomorphism of the form
\begin{equation}\label{pseudoiso}\mathbf{M}^\vee\longrightarrow \left(\bigoplus_{i=1}^s \Lambda/(\varpi^{\mu_i})\right)\oplus \left(\bigoplus_{j=1}^t \Lambda/\left(f_j(T)^{\lambda_j}\right)\right),\end{equation} where $\mu_i$ are positive integers and $f_j(T)$ are irreducible distinguished polynomials. The $\mu$ invariant of $\mathbf{M}$ is defined as follows
\[\mu(\mathbf{M}):=\sum_{i=1}^s \mu_i,\]where we set $\mu(\mathbf{M})=0$ if $s=0$. From the definition of the $\mu$-invariant, it is clear that $\mu(\mathbf{M})=0$ if and only if $M^\vee$ is finitely generated as a $\Z_p$-module.

\begin{proposition}\label{basic mu=0 criterion}
    Let $\mathbf{M}$ be a cofinitely generated and cotorsion $\Lambda$-module. Then, $\mu(\mathbf{M})=0$ if and only if $\mathbf{M}[\varpi]$ is finite.
\end{proposition}
\begin{proof}
    The result is a direct consequence of the structure theorem for $\Lambda$-modules. We have a pseudo-isomorphism 
    \[\mathbf{M}^\vee\longrightarrow \left(\bigoplus_{i=1}^s \Lambda/(\varpi^{\mu_i})\right)\oplus \left(\bigoplus_{j=1}^t \Lambda/\left(f_j(T)^{\lambda_j}\right)\right),\]as described in \eqref{pseudoiso}. Let $\Omega$ denote the mod-$\varpi$ reduction of $\Lambda$. We identify $\left(\mathbf{M}[\varpi]\right)^\vee$ with $\mathbf{M}^\vee/\varpi \mathbf{M}^\vee$. The mod-$\varpi$ reduction of the above map is a pseudo-isomorphism 
    \[\left(\mathbf{M}[\varpi]\right)^\vee\longrightarrow \Omega^s\oplus \left(\bigoplus_{j=1}^t \Omega/(T^{d_j\lambda_j})\right),\] where $d_j=\op{deg}f_j(T)$. Clearly, $\Omega/(T^{d_j\lambda_j})$ is a finite dimensional $\F$-vector space, and $\Omega$ is infinite. Therefore, $M[\varpi]$ is finite if and only if $s=0$. We note that $s=0$ if and only if $\mu(\mathbf{M})=0$, this proves the result.
\end{proof}

\subsection{Selmer groups associated to modular forms and elliptic curves}
\par Let $\tau$ be a variable in the complex upper half plane and set $q:=\op{exp}(2\pi i \tau)$. We take $f=\sum_{n=1}^\infty a_n(f)q^n$ be a normalized Hecke eigencuspform of weight $k\geq 2$. We assume that with respect to the embedding $\iota_p$, the modular form $f$ is ordinary. This means that $\iota_p(a_p(f))$ is a unit of $\cO$. Let $K$ be the extension of $\Q_p$ generated by $\{\iota_p(a_n(f))\mid n\in \Z_{\geq 1}\}$. We note that $K$ is a finite extension of $\Q_p$. Let $\rho_{f, \iota_p}: \op{G}_{\Q}\rightarrow \op{GL}_2(K)$ be the Galois representation associated to $(f, \iota_p)$. Let $V=V_{f, \iota_p}$ be the underlying $2$-dimensional $K$-vector space on which $\op{G}_{\Q}$ acts by $K$-linear automorphisms. We choose a Galois stable $\cO$-lattice $M$ in $V$, and set $A:=V/M$. As a module over $\op{G}_p$, $M$ fits into a short exact sequence 
\[0\rightarrow M_0\rightarrow M\rightarrow M_1\rightarrow 0.\] The modules $M_0$ and $M_1$ are uniquely determined by the property that $M_0\simeq \cO(\alpha \chi^{k-1})$ and $M_1\simeq \cO(\alpha')$, where $\alpha$ and $\alpha'$ are unramified characters. We set $W:=(M_0)\otimes_{\cO}K$ and $C:=\op{image}\left(W\rightarrow A\right)$. Then the Greenberg Selmer group $S_A(\Q_\infty)$ associated to $(A, C)$ clearly satisfies the Assumption \ref{basic ass}. As is well known, in this context, the Greenberg Selmer group is pseudo-isomorphic to the Bloch-Kato Selmer group considered over $\Q_\infty$ (cf. \cite[Corollary 4.3]{ochiai2000control} for further details). It follows from results of Kato \cite{kato2004p} that this Selmer groups are cotorsion over $\Lambda$, i.e., Assumption \ref{cotorsion hyp} is satisfied. We note that the Selmer group $S_A(\Q_\infty)$ depends on the choice of embedding $\iota_p$ and the choice of Galois stable $\cO$-lattice $M$.  
\par Let us now consider Galois representations associated with elliptic curves over $\Q$. Let $E_{/\Q}$ be an elliptic curve with good ordinary reduction at $p$, and let $M:=T_p(E)$ be its $p$-adic Tate-module. The $p$-divisible Galois module $A$ is identified with $E[p^\infty]$, the $p$-power torsion points in $E(\bar{\Q})$. Since $E$ has ordinary reduction at $p$, there is a unique $\Z_p[\op{G}_p]$-submodule $C\simeq \Q_p/\Z_p(\alpha \chi)$ of $E[p^\infty]$, where $\alpha:\op{G}_p\rightarrow \Z_p^\times$ is an unramified character. The quotient $D:=A/C$ is unramified. The Greenberg Selmer group associated to $(A, C)$ is then denoted $\op{Sel}_{p^\infty}(\Q_\infty, E)$. Since $E$ arises from a Hecke eigencuspform of weight $2$, it follows from results of Kato \cite{kato2004p} that $\op{Sel}_{p^\infty}(\Q_\infty, E)^\vee$ is a torsion $\Lambda$-module, i.e., the Assumption \ref{cotorsion hyp} is satisfied. In this setting, there is no ambiguity in the definition of the Selmer group, since the field of coefficients equals $\Q_p$, and the $\Z_p$-Galois module is prescribed to be the $p$-adic Tate module of $E$. The Greenberg Selmer group coincides with the classical Selmer group, where the local conditions are defined via Kummer maps. We refer to \cite[section 2]{greenbergmain} for further details. Throughout, we shall set $\mu_p(E)$ to denote the $\mu$-invariant of the Selmer group $\op{Sel}_{p^\infty}(\Q_\infty, E)$.

\par If $E'$ is another elliptic curve over $\Q$ which is $\Q$-isogenous to $E'$, then, $T_p(E')\otimes_{\Z_p}\Q_p$ is isomorphic to $T_p(E)\otimes_{\Z_p}\Q_p$ as a $\Q_p[\op{G}_\Q]$-module, however, $T_p(E)$ is not isomorphic to $T_p(E')$. It is possible that $\mu_p(E')=0$, while $\mu_p(E)>0$, cf. \cite[section 7]{ray2021mu}.

%\begin{ass}
 %   The dual Selmer group $S_A(\Q_\infty)^\vee$ is a torsion $\Lambda$-module.
%\end{ass}

%The above assumption is known to hold for $p$-ordinary Hecke eigencuspforms of weight $k\geq 2$, due to the work of Kato \cite{kato2004p}.

\section{The residual Selmer group and the vanishing of the $\mu$-invariant}\label{s 3}

\subsection{The residual Selmer group}
Let $M$ be a module over $\op{G}_{\Q}$ for which Assumption \ref{basic ass} is satisfied. Associated with $S_A(\Q_\infty)$ is the residual Selmer group associated to the pair $(A, C)$. Stipulate that the cotorsion Assumption \ref{cotorsion hyp} is also satisfied. Set $A[\varpi^n]$ to denote the kernel of the multiplication by $\varpi^n$ endomorphism of $A$. We denote by $\bar{A}:=A[\varpi]$, and refer to the representation 
\[\rho_{\bar{A}}:\op{G}_\Q\rightarrow \op{Aut}_{\F}(\bar{A})\xrightarrow{\sim} \op{GL}_n(\F)\]as the \emph{residual representation}. This is because we may identify $\bar{A}:=A[\varpi]$ with $M/\varpi M$, and thus think of $\rho_{\bar{A}}$ as the mod-$\varpi$ reduction of $\rho_M$. We let $\bar{C}:=C[\varpi]$, and $\bar{D}:=\bar{A}/\bar{C}$; note that the vector spaces $\bar{A}, \bar{D}$ and $\bar{C}$ are $\F[\op{G}_p]$-modules and they fit into a short exact sequence
\[0\rightarrow \bar{C}\rightarrow \bar{A}\rightarrow \bar{D}\rightarrow 0.\]
We now introduce the \emph{residual Selmer group} associated to $(\bar{A}, \bar{C})$. For $\ell\neq p$, set \[\cH_\ell(\Q_\infty, \bar{A}):=\prod_{\eta|\ell} H^1(\Q_{\infty,\eta}, \bar{A}),\] where $\eta$ runs over all primes of $\Q_\infty$ that lie above $\ell$. At $p$, the local condition is defined by setting $\cH_p(\Q_\infty, \bar{A}):=H^1(\Q_{\infty, \eta_p}, \bar{A})/\bar{L}_{\eta_p}$, where
\[\bar{L}_{\eta_p}:=\op{ker}\left(H^1(\Q_{\infty, \eta_p}, \bar{A})\xrightarrow{\bar{\kappa}_p} H^1(I_{\eta_p}, \bar{D})\right);\] the map $\bar{\kappa}_p$ is the mod-$\varpi$ reduction of $\kappa_p$. 
\begin{definition}
    With respect to notation above, the residual Selmer group is defined as follows
    \[S_{\bar{A}}(\Q_\infty):=\op{ker}\left(H^1(\Q_\Sigma/\Q_\infty, \bar{A})\rightarrow \bigoplus_{\ell\in \Sigma} \cH_\ell(\Q_\infty, \bar{A})\right).\]
\end{definition}
We note in passing that this Selmer group depends not only on the residual representation, but also the choice $\bar{C}$. For an elliptic curve $E_{/\Q}$ with good ordinary reduction at $p$, the space $\bar{A}$ is identified with $E[p]$, and there is a unique one dimensional subspace $\bar{C}$ on which the inertia group at $p$ acts via the mod-$p$ cyclotomic character. Therefore, there is no ambiguity in the definition when it is specialized to an elliptic curve with good ordinary reduction at $p$. We now study the relationship between the residual Selmer group and the $\mu$-invariant of $S_A(\Q_\infty)$.
\begin{lemma}\label{g finite kernel and cokernel lemma}
    There is a natural map \[g:S_{\bar{A}}(\Q_\infty)\rightarrow S_A(\Q_\infty)[\varpi]\] with finite kernel and cokernel.
\end{lemma}
\begin{proof}
   The result follows from \cite[Section 4.2]{emerton2006variation}, however, for the benefit of the reader we provide a sketch of this insightful argument. Recall that $\bar{A}=A[\varpi]$, consider the Kummer sequence of $\Z_p[\op{G}_\Q]$-modules
    \begin{equation}\label{kummer seq}0\rightarrow \bar{A}\rightarrow A\xrightarrow{\times \varpi} A\rightarrow 0.\end{equation}
    This induces and exact sequence 
    \begin{equation}\label{kummer eqn 1}
    0\rightarrow \left(\frac{H^0\left(\Q_\infty, A\right)}{\varpi H^0\left(\Q_\infty, A\right)}\right)\rightarrow H^1(\Q_\Sigma/\Q_\infty, \bar{A})\xrightarrow{f} H^1(\Q_\Sigma/\Q_\infty, A)[\varpi]\rightarrow 0.
    \end{equation}
 Let $\ell$ be a prime, and $\eta$ be a prime of $\Q_\infty$ that lies above $\ell$. From the Kummer sequence \eqref{kummer seq}, we obtain an exact sequence 
 \begin{equation}\label{kummer eqn 2}
    0\rightarrow \left(\frac{H^0\left(\Q_{\infty,\eta}, A\right)}{\varpi H^0\left(\Q_{\infty,\eta}, A\right)}\right)\rightarrow H^1(\Q_{\infty, \eta}, \bar{A})\xrightarrow{f_\eta} H^1(\Q_{\infty,\eta}, A)[\varpi]\rightarrow 0.
    \end{equation}
    For $\ell\neq p$, we let 
    \[f_\ell: \cH_\ell(\Q_\infty, \bar{A})\rightarrow  \cH_\ell(\Q_\infty, A)[\varpi]\]be the product of the maps $f_\eta$, where $\eta$ ranges over the primes above $\ell$. Since $\ell$ is finitely decomposed in $\Q_\infty$, it follows from \eqref{kummer eqn 2} that the kernel of $f_\ell$. Consider the commutative square
 \[\begin{tikzcd}
H^1(\Q_{\infty, \eta_p}, A[\varpi]) \arrow[r, "\bar{\kappa}_p"] \arrow[d]
& H^1(I_{\eta_p}, D[\varpi]) \arrow[d] \\
H^1(\Q_{\infty, \eta_p}, A) \arrow[r, "\kappa_p"]
& H^1(I_{\eta_p}, D).
\end{tikzcd}\]
We identify $\cH_p(\Q_\infty, A)$ (resp. $\cH_p(\Q_\infty, \bar{A})$) with the image of $\kappa_p$ (resp. $\bar{\kappa}_p$). Form the commutativity of the above square, we obtain a map
\[f_p: \cH_p(\Q_\infty, \bar{A})\rightarrow \cH_p(\Q_\infty, A)[\varpi].\] From the exact sequence of $I_{\eta_p}$-modules
\[0\rightarrow D[\varpi]\rightarrow D\xrightarrow{\times \varpi} D\rightarrow 0, \] we obtain an exact sequence
\[0\rightarrow \frac{H^0(I_{\eta_p}, D)}{\varpi H^0(I_{\eta_p}, D)}\rightarrow H^1(I_{\eta_p}, D[\varpi])\rightarrow H^1(I_{\eta_p}, D)\rightarrow 0.\]
We note that $D$ is divisible and unramified at $\eta_p$, hence, $H^0(I_{\eta_p}, D)=D$. Since $D$ is divisible, it follows that the map
\[H^1(I_{\eta_p}, D[\varpi])\rightarrow H^1(I_{\eta_p}, D)\] is injective, and hence $f_p$ is injective. The map $f$ restricts to a map 
\[g:S_{\bar{A}}(\Q_\infty)\rightarrow S_A(\Q_\infty)[\varpi]\] which fits into a commutative diagram
 \[\begin{tikzcd}
0 \arrow[r] & S_{\bar{A}}(\Q_\infty) \arrow[r] \arrow[d, "g"]
& H^1(\Q_\Sigma/\Q_\infty, \bar{A}) \arrow[r, "\Phi_{\bar{A}}"]\arrow[d, "f"] & \bigoplus_{\ell\in \Sigma} \cH_\ell(\Q_\infty, \bar{A}) \arrow[d, "h=\oplus f_\ell"]\\
0 \arrow[r] & S_A(\Q_\infty)[\varpi] \arrow[r]
& H^1(\Q_\Sigma/\Q_\infty, A)[\varpi] \arrow[r, "\Phi_{A}"] & \bigoplus_{\ell\in \Sigma} \cH_\ell(\Q_\infty, A)[\varpi].
\end{tikzcd}\]
Let $h'$ be the restriction of $h$ to the image of $\Phi_V$. From the snake lemma, we obtain an exact sequence
\[0\rightarrow \ker g\rightarrow \ker f \rightarrow \ker h'\rightarrow \coker g \rightarrow \coker f=0.\]
From \eqref{kummer eqn 1}, we know that the kernel of $f$ is finite, and hence, the kernel of $g$ is finite. We have shown that the kernel of $h$ is finite, and hence, $\ker h'$ is finite. Therefore, we find that both the kernel and cokernel of $g$ are finite, and this completes the proof.
\end{proof}

\begin{proposition}\label{prop 3.3}
Suppose that the Assumptions \ref{basic ass} and \ref{cotorsion hyp} are satisfied. Then, the following conditions are equivalent.
\begin{enumerate}
    \item The $\mu$-invariant of $S_A(\Q_\infty)$ is equal to $0$.
    \item The residual Selmer group $S_{\bar{A}}(\Q_\infty)$ is finite.
\end{enumerate}
\end{proposition}
\begin{proof}
   It follows from Proposition \ref{basic mu=0 criterion} that the $\mu$-invariant of $S_A(\Q_\infty)$ is $0$ if and only if $S_A(\Q_\infty)[\varpi]$ is finite. Then it follows from Lemma \ref{g finite kernel and cokernel lemma} that $S_A(\Q_\infty)[\varpi]$ is finite if and only if $S_{\bar{A}}(\Q_\infty)$ is finite. This completes the proof.
\end{proof}

\section{The fine Selmer group}\label{s 4}
In this section, we recall the definition of the fine Selmer group associated to $A$. For further details, we refer to \cite{coates2005fine,DRS2023}. We do \emph{not} insist that $A$ satisfies the Assumption \ref{basic ass} in this section. Recall that $\Sigma$ is a finite set of prime numbers containing $p$ and the primes that are ramified in $A$. Let $F$ be a number field and $F_\infty$ be the composite of $F$ with $\Q_\infty$. For any prime $\ell$, set $\cK_\ell(F_\infty, A):=\prod_{\eta|\ell} H^1(F_{\infty,\eta}, A)$, where $\eta$ runs over the primes of $F_\infty$ that lie above $\ell$. We note that for any prime $\ell\neq p$, the local condition $\cK_\ell(\Q_\infty, A)$ coincides with $\cH_\ell(\Q_\infty, A)$, the difference lies at the prime $p$. The \emph{fine Selmer group} is defined as follows
\[S_A^0(F_\infty):=\op{ker}\left(H^1(F_\Sigma/F_\infty, A)\rightarrow \bigoplus_{\ell\in \Sigma} \cK_\ell(F_\infty, A)\right).\] Here, $F_\Sigma$ is the maximal extension of $F$ in which all primes $\ell\notin \Sigma$ are unramified. As is well known, the definition above is independent of the choice of $\Sigma$. For further details, we refer to \cite[Lemma 3.2]{SujathaWitte}. Given an elliptic curve $E$ over a number field $F$, and a prime number $p$, set $\op{Sel}_{p^\infty}^0(F_\infty, E)$ be the fine Selmer group associated to $A=E[p^\infty]$ over $F_{\infty}$.

\par Define the \emph{residual fine Selmer group} by setting 
\[\cK_\ell(F_\infty, \bar{A}):=\prod_{\eta|\ell} H^1(F_{\infty,\eta}, \bar{A})\] for all prime numbers $\ell\in \Sigma$, and setting 
\[S_{\bar{A}}^0(F_\infty):=\op{ker}\left(H^1(F_\Sigma/F_\infty, \bar{A})\rightarrow \bigoplus_{\ell\in \Sigma} \cK_\ell(F_\infty, \bar{A})\right).\]
The fine Selmer group fits into a left exact sequence
\[0\rightarrow S_A^0(\Q_\infty)\rightarrow S_A(\Q_\infty)\rightarrow L_{\eta_p}. \]
\begin{lemma}\label{g' finite kernel and cokernel lemma}
    There is a natural map \[g_0:S_{\bar{A}}^0(\Q_\infty)\rightarrow S_A^0(\Q_\infty)[\varpi]\] with finite kernel and cokernel.
\end{lemma}
\begin{proof}
    The proof is similar to that of Lemma \ref{g finite kernel and cokernel lemma}, we provide a sketch of the details. The map $g_0$ is induced by restricting $f$ to $S_{\bar{A}}^0(\Q_\infty)$. It is easy to see that the image of this restriction lies in $S_A^0(\Q_\infty)[\varpi]$. The map $g_0$ fits into a natural commutative diagram depicted below
     \[\begin{tikzcd}
0 \arrow[r] & S_{\bar{A}}^0(\Q_\infty) \arrow[r] \arrow[d, "g_0"]
& H^1(\Q_\Sigma/\Q_\infty, \bar{A}) \arrow[r, "\Phi_{\bar{A}}' "]\arrow[d, "f"] & \bigoplus_{\ell\in \Sigma} \mathcal{K}_\ell(\Q_\infty, \bar{A}) \arrow[d]\\
0 \arrow[r] & S_A^0(\Q_\infty)[\varpi] \arrow[r]
& H^1(\Q_\Sigma/\Q_\infty, A)[\varpi] \arrow[r, "\Phi_{A}' "] & \bigoplus_{\ell\in \Sigma} \mathcal{K}_\ell(\Q_\infty, A)[\varpi].
\end{tikzcd}\]
In the above diagram, the horizontal maps $\Phi_{\bar{A}}'$ and $\Phi_A'$ are induced by restriction maps. The kernels of both vertical maps in the diagram are finite. By the same argument as in the proof of Lemma \ref{g finite kernel and cokernel lemma}, it follows that the kernel and cokernel of $g_0$ are finite.
\end{proof}

\begin{proposition}\label{prop 4.2}
With respect to notation above, the following conditions are equivalent.
\begin{enumerate}
    \item The $\mu$-invariant of $S_A^0(\Q_\infty)$ is equal to $0$.
    \item The residual fine Selmer group $S_{\bar{A}}^0(\Q_\infty)$ is finite.
\end{enumerate}
\end{proposition}
\begin{proof}
   It follows from Proposition \ref{basic mu=0 criterion} that the $\mu$-invariant of $S_A^0(\Q_\infty)$ is $0$ if and only if $S_A^0(\Q_\infty)[\varpi]$ is finite. Then it follows from Lemma \ref{g' finite kernel and cokernel lemma} that $S_A^0(\Q_\infty)[\varpi]$ is finite if and only if $S_{\bar{A}}^0(\Q_\infty)$ is finite. This completes the proof.
\end{proof}

At this point, it is pertinent to recall a conjecture of Coates and Sujatha on the structure of the fine Selmer group associated with an elliptic curve. For futher details, see \cite[Conjecture A]{coates2005fine}.

\begin{conjecture}[Coates-Sujatha]\label{CS conj}
    Let $E$ be an elliptic curve over a number field $F$ and $p$ be a prime above which $E$ has good reduction. Then, the fine Selmer group $\op{Sel}_{p^\infty}^0(F_\infty, E)$ is cofinitely generated as a $\Z_p$-module.
\end{conjecture}

Under some additional conditions, the above conjecture is known to hold. Let $F(E_{p^\infty})$ be the Galois extension of $F$ generated by $E_{p^\infty}$. In greater detail, letting $\rho_{E,p}: \op{G}_F\rightarrow \op{GL}_2(\Z_p)$ be the Galois representation on the $p$-adic Tate module of $E$, the extension $F(E_{p^\infty}):=\bar{F}^{\ker \rho_{E,p}}$. Note that $\rho_{E,p}$ induces an inclusion of $\op{Gal}(F(E_{p^\infty})/F)$ into $\op{GL}_2(\Z_p)$.

\begin{theorem}[Coates-Sujatha]\label{CS main result}
Let $E_{/F}$ be an elliptic curve and $p$ an odd prime such that $F(E_{p^\infty})$ is a pro-$p$ extension of $F$. Then, the following conditions are equivalent
\begin{enumerate}
    \item Conjecture \ref{CS conj} is valid, i.e., $\op{Sel}_{p^\infty}^0(F_\infty, E)$ is cofinitely generated as a $\Z_p$-module. 
    \item The classical Iwasawa $\mu$-invariant $\mu_p(F)$ vanishes.
\end{enumerate}
\end{theorem}
\begin{proof}
    The above result is \cite[Theorem 3.4]{coates2005fine}.
\end{proof}

\section{Structure of the residual Greenberg Selmer group}\label{s 5}
In this section, we prove some of the main results of the article which will be of key importance in the proof of Theorem \ref{main thm}. At the end of this section, we shall prove Theorem \ref{main thm 2}. We begin by proving an explicit relationship between the residual Selmer group and the residual fine Selmer group. These residual Selmer groups were introduced in the previous section.

It is necessary to introduce an assumption on the Galois action on the residual representation $\bar{A}$. 
\begin{ass}\label{main assumption}
    Assume that $\bar{C}$ does not contain a non-zero $\op{G}_{\Q_\infty}$-submodule of $\bar{A}$. 
\end{ass}
For Galois representations associated with elliptic curves, we characterize precisely when the above Assumption holds.
\begin{proposition}\label{assumption ell curve}
    Let $E_{/\Q}$ be an elliptic curve with good ordinary reduction at $p$, and let $\bar{\rho}_{E, p}:\op{G}_\Q\rightarrow \op{GL}_2(\F_p)$ be the Galois representation on the torsion subgroup $E[p]\subset E(\bar{\Q})$. Thus, $A=E[p^\infty]$ and $\bar{A}=E[p]$. Let $(e_1, e_2)$ be an ordered basis of $E[p]$ such that $\bar{C}=\F_p \cdot e_1$. The following assertions hold.
    
    \begin{enumerate}
        \item\label{p1 assumption ell curve} Assume that $\bar{\rho}_{E, p}$ is irreducible. Then, the Assumption \ref{main assumption} holds. 
        \item\label{p2 assumption ell curve} Assume that $\bar{\rho}_{E, p}$ is reducible and indecomposable\footnote{i.e., it is reducible but does not split into a sum of characters with respect to any basis}, and with respect to the basis $(e_1, e_2)$, takes the form
        \[\bar{\rho}_{E, p}=\mtx{\varphi_1}{0}{\ast}{\varphi_2}.\]
        Then, the Assumption \ref{main assumption} holds.
        \item\label{p3 assumption ell curve} Assume that $\bar{\rho}_{E, p}$ is reducible and with respect to the basis $(e_1, e_2)$, takes the form
        \[\bar{\rho}_{E, p}=\mtx{\varphi_1}{\ast}{0}{\varphi_2}.\] Then, the Assumption \ref{main assumption} does not hold.
    \end{enumerate} 
\end{proposition}
\begin{proof}
    \par We begin with part \eqref{p1 assumption ell curve}. We set 
    \[\mathcal{G}:=\bar{\rho}_{E,p}(\op{G}_{\Q})\text{ and }\mathcal{H}:=\bar{\rho}_{E,p}(\op{G}_{\Q_\infty}).\] Since $\Q_\infty/\Q$ is a $\Z_p$-extension, it follows that $\cH$ is a normal subgroup of $\cG$ of index $|\cG/\cH|=p^t$, where $t\in \Z_{\geq 0}$. If $\mathcal{G}=\mathcal{H}$, then, $\bar{\rho}_{E,p}$ remains irreducible when restricted to $\op{G}_{\Q_\infty}$. Therefore, in this case, $\bar{C}$ is not a $\op{G}_{\Q_\infty}$-submodule and the Assumption \ref{main assumption} holds. Therefore, we assume that $p$ divides $|\cG/\cH|$. Since $p$ divides $\cG$, it follows that either $\cG$ contains $\op{SL}_2(\F_p)$, or $\cG$ is contained in a Borel subgroup of $\op{GL}_2(\F_p)$ (cf. \cite[Proposition 3.1]{sutherland2016computing}). Since $\bar{\rho}_{E,p}$ is irreducible, $\cG$ is not contained in a Borel subgroup. Hence, $\cG$ contains $\op{SL}_2(\F_p)$. Suppose that $\cH$ is contained in a Borel subgroup; then we find that $|\cH|$ divides $p(p-1)^2$. On the other hand, since $\cG$ contains $\op{SL}_2(\F_p)$, it follows that $|\cG|$ is divisible by $ |\op{SL}_2(\F_p)|=(p^2-1)p$. Since $|G|=p^t |H|$, it follows that $(p^2-1)$ divides $p(p-1)^2$, which is a contradiction. Therefore, $\cH$ is not contained in a Borel subgroup and hence, the representation $\bar{\rho}_{E,p}$ is irreducible when restricted to $\op{G}_{\Q_\infty}$. Therefore, $\bar{A}$ does not contain any non-zero proper $\op{G}_{\Q_\infty}$ submodules, and this completes the proof of part \eqref{p1 assumption ell curve}.

    \par For the proof of part \eqref{p2 assumption ell curve} it suffices to show that $\bar{\rho}_{E,p}$ remains indecomposable even after restriction to $\op{G}_{\Q_\infty}$. We define a function $\beta: \op{Gal}(\Q_\Sigma/\Q)\rightarrow \F_p$ by setting $\beta(g)$ to denote the lower left entry of $\bar{\rho}_{E,p}$. It is easy to see that $\beta$ gives rise to a cocycle \[\beta \in Z^1\left(\Q_{\Sigma}/\Q, \F_p(\varphi_2\varphi_1^{-1})\right).\] Let $[\beta]\in H^1\left(\Q_{\Sigma}/\Q, \F_p(\varphi_2\varphi_1^{-1})\right)$ denote the corresponding cohomology class. Since it is assumed that $\bar{\rho}_{E,p}$ is indecomposable, it follows that $[\beta]$ is a non-zero cohomology class. In order to show that the restriction 
    \[\bar{\rho}_{E,p}:\op{Gal}(\Q_\Sigma/\Q_\infty)\rightarrow \op{GL}_2(\F_p)\]is indecomposable, it suffices to show that the restriction of $[\beta]$ to $H^1\left(\Q_{\Sigma}/\Q_\infty, \F_p(\varphi_2\varphi_1^{-1})\right)$ is non-zero. From the inflation restriction sequence, the kernel of the restriction map
    \begin{equation}\label{restriction map}H^1\left(\Q_{\Sigma}/\Q, \F_p(\varphi_2\varphi_1^{-1})\right)\rightarrow H^1\left(\Q_{\Sigma}/\Q_\infty, \F_p(\varphi_2\varphi_1^{-1})\right)\end{equation}
    is $H^1\left(\Q_\infty/\Q, \left(\F_p(\varphi_2\varphi_1^{-1})\right)^{\op{G}_{\Q_\infty}}\right)$. Since $\varphi_1$ is ramified at $p$ and $\varphi_2$ is unramified at $p$, we find that $\varphi_2\varphi_1^{-1}\neq 1$. Since $\op{Gal}(\Q_\infty/\Q)$ is a pro-$p$ extension and the character $\varphi_2\varphi_1^{-1}$ takes values in $\F_p^\times$, we find that the restriction of $\varphi_2\varphi_1^{-1}$ to $\op{G}_{\Q_\infty}$ is non-trivial. Therefore, $\left(\F_p(\varphi_2\varphi_1^{-1})\right)^{\op{G}_{\Q_\infty}}=0$ and the restriction map \eqref{restriction map} is injective. Hence, the restriction of $\beta$ to $H^1\left(\Q_{\Sigma}/\Q_\infty, \F_p(\varphi_2\varphi_1^{-1})\right)$ is non-zero. This proves that 
    \[\bar{\rho}_{E,p|\Q_\infty}:\op{G}_{\Q_\infty}\rightarrow \op{GL}_2(\F_p)\]is indecomposable. We therefore have shown that $\bar{C}$ is not a $\op{G}_{\Q_\infty}$-submodule of $\bar{A}$. This completes the proof of part \eqref{p2 assumption ell curve}.
    \par For part \eqref{p3 assumption ell curve}, observe that $\bar{C}$ is a $\op{G}_{\Q}$-submodule of $\bar{A}$. In particular, it is a $\op{G}_{\Q_\infty}$ submodule of $\bar{A}$, and the Assumption \ref{main assumption} is not satisfied.
\end{proof}

The residual Selmer group $S_{\bar{A}}^0(\Q_\infty)$ is contained in $S_{\bar{A}}(\Q_\infty)$; set \[\bar{S}_{\bar{A}}(\Q_\infty):=\frac{S_{\bar{A}}(\Q_\infty)}{S_{\bar{A}}^0(\Q_\infty)}.\] 
We postpone the proof of the above result till the end of this section. First, we introduce some further notation. Let $L=\Q(\bar{A})$ be the field cut out by the residual representation. In other words, $L$ is the field $\bar{\Q}^{\op{ker}\rho_{\bar{A}}}$, the field fixed by the kernel of the residual representation $\rho_{\bar{A}}$. We note that $\Q(\bar{A})$ is a finite Galois extension of $\Q$ and the Galois group $\op{Gal}(\Q(\bar{A})/\Q)$ is naturally isomorphic to the image of $\rho_{\bar{A}}$;  the representation $\rho_{\bar{A}}$ induces an isomorphism 
\[\op{Gal}(\Q(\bar{A})/\Q)\xrightarrow{\sim} \op{image}\rho_{\bar{A}}.\]
We observe that $\op{G}_L$ is the kernel of $\rho_{\bar{A}}$, and hence acts trivially on $\bar{A}$. Let $L_\infty:=L\cdot \Q_\infty$ be the cyclotomic $\Z_p$-extension of $L$. Let $\beta$ be the prime of $L_\infty$ that lies above $p$ that coincides with the choice of embedding $\iota_p$, and denote by $I_{\beta}$ the inertia group at $\beta$. Note that $I_{\beta}$ is contained in the inertia group $I_{\eta_p}$.  

\par We define a Selmer group $S_{\bar{A}}(L_\infty)$ associated to $(\bar{A}, \bar{C})$ over $L_\infty$. For each prime number $\ell$, we define a local condition $\cH_\ell(L_\infty, \bar{A})$. For $\ell\neq p$, set 
\[\cH_\ell(L_\infty, \bar{A}):=\prod_{\eta|\ell} H^1(L_{\infty, \eta}, \bar{A}),\]
where $\eta$ runs through all primes of $L_\infty$ that lies above $\ell$. We note that this is a finite set of primes. At the prime $p$, we set 
\[\cH_p(L_\infty, \bar{A}):=\left(\frac{H^1(L_{\infty, \beta}, \bar{A})}{\bar{L}_{\beta}}\right),\] where 
\[\bar{L}_{\beta}:=\op{ker}\left(H^1(L_{\infty, \beta}, \bar{A})\rightarrow H^1(I_{\beta}, \bar{D})\right).\]
Note that since $\rho_{\bar{A}}$ is unramified outside $\Sigma$, and hence, $L$ is contained in $\Q_{\Sigma}$. With respect to notation above, the residual Selmer group is defined as follows
    \[S_{\bar{A}}(L_\infty):=\op{ker}\left(H^1(\Q_\Sigma/L_\infty, \bar{A})\rightarrow \bigoplus_{\ell\in \Sigma} \cH_\ell(L_\infty, \bar{A})\right).\]
 We relate the two residual Selmer groups $S_{\bar{A}}(\Q_\infty)$ and $S_{\bar{A}}(L_\infty)$. We shall set $G:=\op{Gal}(L_\infty/\Q_\infty)$. We note that $\op{Gal}(\bar{\Q}/L)$ is the kernel of $\rho_{\bar{A}}$ and therefore, the Galois action of $\op{Gal}(\bar{\Q}/L)$ on $\bar{A}$ is trivial. We identify $H^1(\Q_\Sigma/L_\infty, \bar{A})$ with the group of homomorphisms $\op{Hom}\left(\op{Gal}(\Q_\Sigma/L_\infty), \bar{A}\right)$. For $g\in G$, take $\tilde{g}\in \op{Gal}(\Q_\Sigma/\Q_\infty)$ to be a lift of $g$. Take $\psi\in \op{Hom}\left(\op{Gal}(\Q_\Sigma/L_\infty), \bar{A}\right)$, we note that since $\bar{A}$ is abelian, $\psi(\tilde{g} x \tilde{g}^{-1})$ is independent of the choice of lift $\tilde{g}$. Define an action of $G$ on $\op{Hom}\left(\op{Gal}(\Q_\Sigma/L_\infty), \bar{A}\right)$, by setting
\[(g\cdot \psi)(x):=g^{-1}\psi(\tilde{g}x\tilde{g}^{-1}).\] Therefore, a homomorphism $\psi$ in $\op{Hom}\left(\op{Gal}(\Q_\Sigma/L_\infty), \bar{A}\right)^G$ is one which is $G$-equivariant, in the sense that 
\[\psi(\tilde{g}x\tilde{g}^{-1})=g\psi(x).\]Consider the inflation-restriction sequence
\begin{equation}\label{inf res seq}0\rightarrow H^1(G, \bar{A})\xrightarrow{inf} H^1(\Q_{\Sigma}/\Q_\infty, \bar{A})\xrightarrow{res} \op{Hom}\left(\op{Gal}(\Q_\Sigma/L_\infty), \bar{A}\right)^G.\end{equation}
 The restriction map 
    \[\op{res}: H^1(\Q_{\Sigma}/\Q_\infty, \bar{A})\rightarrow H^1(\Q_{\Sigma}/L_\infty, \bar{A})\] induces a map 
    \[\op{res}: S_{\bar{A}}(\Q_\infty)\rightarrow S_{\bar{A}}(L_\infty).\] Since $G$ is finite, $H^1(G,\bar{A})$ is finite, and thus the kernel of this restriction map is finite.

    \par We let $S_{\bar{A}}^{\op{nr}}(L_\infty)$ be the subspace of $S_{\bar{A}}(L_\infty)$ consisting of the classes that are unramified at $\beta$. Note that $S_{\bar{A}}(L_\infty)$ consists of homomorphisms 
    \[\psi:\op{Gal}(\Q_\Sigma/L_\infty)\rightarrow \bar{A}\] that satisfy the following conditions
    \begin{enumerate}
        \item $\psi$ trivial when restricted to the decomposition group of any prime $\eta$ of $L_\infty$ that lies above a prime $\ell\in \Sigma\backslash\{p\}$,
        \item $\psi(I_{\beta})$ is contained in $\bar{C}$.
    \end{enumerate}  
    The subset $S_{\bar{A}}^{\op{nr}}(L_\infty)$ consists of those classes for which $\psi(I_{\beta})=0$.

\begin{conjecture}\label{main conjecture generalized}
    Suppose that Assumption \ref{main assumption} holds, then, the image of the restriction map 
    \[ S_{\bar{A}}(L_\infty)^G\rightarrow \Hom\left(I_{\beta}, \bar{C}\right)\] is finite.
\end{conjecture}

\begin{theorem}\label{main 1}
    Let $(A, C)$ be such that Assumption \ref{main assumption} holds for $(\bar{A}, \bar{C})$. Furthermore, assume that the Conjecture \ref{main conjecture generalized} is also satisfied. Then, the following assertions hold
    \begin{enumerate}
        \item\label{p1 main 1} $\bar{S}_{\bar{A}}(\Q_\infty):=\frac{S_{\bar{A}}(\Q_\infty)}{S_{\bar{A}}^0(\Q_\infty)}$ is finite.
        \item\label{p2 main 1} The $\mu$-invariant of $S_A(\Q_\infty)$ vanishes if and only if the $\mu$-invariant of $S_A^0(\Q_\infty)$ vanishes.
    \end{enumerate}
\end{theorem}

\begin{proof}[Proof of Theorem \ref{main 1}]
\par We shall set $S_{\bar{A}}^{\op{nr}}(\Q_\infty)$ to consist of all classes $f\in S_{\bar{A}}^{\op{nr}}(\Q_\infty)$ that are unramified at $\eta_p$. It is easy to see that $S_{\bar{A}}^{0}(\Q_\infty)$ is of finite index in $S_{\bar{A}}^{\op{nr}}(\Q_\infty)$. We begin by proving part \eqref{p1 main 1}. We have a short exact sequence 
\begin{equation}\label{ses 5.4}0\rightarrow S_{\bar{A}}^{\op{nr}}(\Q_\infty)\rightarrow S_{\bar{A}}(\Q_\infty)\rightarrow H^1(I_{\eta_p}, \bar{C}).\end{equation} Consider the commutative square

\[\begin{tikzcd}
 S_{\bar{A}}(\Q_\infty) \arrow[r] \arrow[d]
& H^1(I_{\eta_p}, \bar{C}) \arrow[d] \\
S_{\bar{A}}(L_\infty)^G \arrow[r]
& H^1(I_{\beta}, \bar{C}).
\end{tikzcd}\]
It follows from Conjecture \ref{main conjecture} that the image of the composed map

\[S_{\bar{A}}(\Q_\infty)\rightarrow H^1(I_{\beta}, \bar{C})\] is finite. Since $I_{\beta}$ has finite index in $I_{\eta_p}$, it follows that the kernel of the restriction map 
\[H^1(I_{\eta_p}, \bar{C})\rightarrow H^1(I_{\beta}, \bar{C})\] is finite. Therefore, we find that the image of 
\[S_{\bar{A}}(\Q_\infty)\rightarrow H^1(I_{\eta_p}, \bar{C})\] is finite. 
From the exact sequence \eqref{ses 5.4}, we deduce that $S_{\bar{A}}^{\op{nr}}(\Q_\infty)$ is of finite index in $S_{\bar{A}}(\Q_\infty)$. Therefore, $S_{\bar{A}}^{0}(\Q_\infty)$ is of finite index in $S_{\bar{A}}(\Q_\infty)$, and the statement of part \eqref{p1 main 1} follows from this. 
\par It follows from part \eqref{p1 main 1} that $S_{\bar{A}}(\Q_\infty)$ is finite if and only if $S_{\bar{A}}^0(\Q_\infty)$ is finite. Proposition \ref{prop 3.3} asserts that $S_{\bar{A}}(\Q_\infty)$ is finite if and only if the $\mu$-invariant of $S_{A}(\Q_\infty)$ is $0$. On the other hand, Proposition \ref{prop 4.2} asserts that $S_{\bar{A}}^0(\Q_\infty)$ is finite if and only if the $\mu$-invariant of $S_{A}^0(\Q_\infty)$ is $0$. Hence, the $\mu$-invariant of $S_{A}(\Q_\infty)$ is $0$ if and only if the $\mu$-invariant of $S_{A}^0(\Q_\infty)$ is $0$. This proves part \eqref{p2 main 1}.
\end{proof}

\begin{corollary}\label{cor 5.5}
    Let $E_{/\Q}$ be an elliptic curve with good ordinary reduction at an odd prime $p$. Assume that $\bar{\rho}_{E,p}$ is irreducible and Conjecture \ref{main conjecture generalized} is satisfied.. Then, the following are equivalent.
    \begin{enumerate}
        \item The $\mu$-invariant of $\op{Sel}_{p^\infty}(\Q_\infty, E)$ vanishes, i.e., Greenberg's conjecture holds.
        \item The $\mu$-invariant of $\op{Sel}_{p^\infty}^0(\Q_\infty, E)$ vanishes, i.e., the Conjecture \ref{CS conj} holds.
    \end{enumerate}
\end{corollary}
\begin{proof}
    Since $\bar{\rho}_{E, p}$ is irreducible, Proposition \ref{assumption ell curve} shows that the Assumption \ref{main assumption} is satisfied. The result therefore follows from part \eqref{p2 main 1} of Theorem \ref{main 1}.
\end{proof}
Let $E_{/\Q}$ be an elliptic curve and $p$ a prime at which $E$ has good ordinary reduction. Let \[\bar{\rho}_{E, p}:\op{G}_{\Q}\rightarrow \op{Aut}(E[p])\xrightarrow{\sim} \op{GL}_2(\F_p)\] be the residual representation on $E[p]$. The splitting field $\Q(E[p])$ is the field extension of $\Q$ which is fixed by the kernel of $\bar{\rho}_{E, p}$. 
\begin{proof}[Proof of Theorem \ref{main thm 2}]
    Since $\bar{\rho}_{E, p}$ is irreducible, it follows from Corollary \ref{cor 5.5} that the $\mu$-invariant of $\op{Sel}_{p^\infty}(\Q_\infty, E)$ is equal to $0$ if and only if the $\mu$-invariant of $\op{Sel}_{p^\infty}^0(\Q_\infty, E)$ is equal to $0$. 
    \par Consider the Galois representation 
    \[\rho_{E, p}:\op{G}_{\Q}\rightarrow \op{GL}_2(\Z_p)\] associated with the $p$-adic Tate module of $E$. The restriction of $\rho_{E, p}$ to $\op{G}_L$ is trivial modulo $p$. This is because $L$ is the splitting field $\Q(E[p]):=\bar{\Q}^{\ker\bar{\rho}_{E,p}}$, and $\op{G}_L$ is the kernel of $\bar{\rho}_{E,p}=\rho_{E,p}\mod{p}$. Therefore, the representation $\rho_{E,p}$ identifies $\op{Gal}\left(L(E_{p^\infty})/L\right)$ with a subgroup of 
    \[\widehat{\op{GL}_2}(\Z_p):=\ker\{\op{GL}_2(\Z_p)\rightarrow \op{GL}_2(\Z/p\Z)\}.\] It is easy to see that $\widehat{\op{GL}_2}(\Z_p)$ is a pro-$p$ group.\footnote{The author is willing to provide further details in support of this claim (if the referee insists).} Hence, the Galois group $\op{Gal}\left(L(E_{p^\infty})/L\right)$is a pro-$p$ group. Since it is assumed that the classical Iwasawa $\mu$-invariant $\mu_p(L)$ vanishes, it follows from Theorem \ref{CS main result} that $\op{Sel}_{p^\infty}^0(L_\infty, E)$ is cofinitely generated as a $\Z_p$-module. In other words, $\op{Sel}_{p^\infty}^0(L_\infty, E)$ is a cotorsion $\Lambda$-module whose $\mu$-invariant vanishes. It is easy to see that the kernel of the natural restriction map
    \[\op{Sel}_{p^\infty}^0(\Q_\infty, E)\rightarrow \op{Sel}_{p^\infty}^0(L_\infty, E)\]is cofinitely generated as a $\Z_p$-module, and hence the $\mu$-invariant of $\op{Sel}_{p^\infty}^0(\Q_\infty, E)$ is $0$. This completes the proof.
\end{proof}

\section{Residually reducible Galois representations arising from elliptic curves and Greenberg's conjecture}\label{s 6}

\par Throughout this section, we fix and elliptic curve $E_{/\Q}$ and an odd prime $p$ at which $E$ has good ordinary reduction. Let $M$ denote the $p$-adic Tate module of $E$. Recall that $\bar{A}$ is the mod-$p$ reduction of $M$, which we may identify with $E[p]$. The module $\bar{C}$ is the $1$-dimensional $\op{G}_p$-submodule which is ramified, and the quotient $\bar{D}:=\bar{A}/\bar{C}$ is unramified. The residual representation $\bar{\rho}_{E,p}=\rho_{\bar{A}}$ is the representation of $\op{G}_\Q$ on $\bar{A}$. We shall assume throughout this section that $\bar{A}$ is reducible as a Galois module. Choose a basis $(e_1, e_2)$ of $\bar{A}$ such that $\bar{C}=\F_p\cdot e_1$. Call such a basis \emph{admissible}; note that for any other admissible basis $(e_1', e_2')$, there are constants $c_1, c_2\in \F_p^\times$ and $d\in \F_p$ for which 
\[e_1'=c_1 e_1\text{ and }e_2'=c_2 e_2+de_1.\]
With respect to an admissible basis $(e_1, e_2)$, the restriction of $\bar{\rho}_{E,p}$ to the decomposition group at $p$ takes the form
\[\bar{\rho}_{E,p|\op{G}_p}=\mtx{\alpha\bar{\chi}}{\ast}{0}{\alpha^{-1}},\]
where $\alpha:\op{G}_p\rightarrow \F_p^\times$ is an unramified character and $\bar{\chi}$ is the mod-$p$ cyclotomic character. There are 3 possibilities for the representation $\bar{\rho}_{E,p}$. These are described below and all matrices are written with respect to an admissible basis $(e_1, e_2)$.
\begin{description}
    \item[Type 1] The representation $\bar{\rho}_{E, p}$ is upper triangular of the form $\mtx{\varphi_1}{\ast}{0}{\varphi_2}$, where $\varphi_1$ is odd and $\varphi_2$ is even.
     \item[Type 2] The representation $\bar{\rho}_{E, p}$ is upper triangular of the form $\mtx{\varphi_1}{\ast}{0}{\varphi_2}$, where $\varphi_1$ is even and $\varphi_2$ is odd.
     \item[Type 3] The representation $\bar{\rho}_{E, p}$ is indecompasable and lower triangular of the form $\mtx{\varphi_1}{0}{\ast}{\varphi_2}$. In this context, to be indecomposable means that there is no admissible basis with respect to which $\bar{\rho}_{E, p}$ is a direct sum of characters.
\end{description}
We note that $\varphi_{1|\op{G}_p}=\alpha\chi$ and $\varphi_{2|\op{G}_p}\simeq \alpha^{-1}$. Note that the Conjecture \ref{main conjecture generalized} specializes to the Conjecture \ref{main conjecture}.

\par The vanishing of the $\mu$-invariant of $\op{Sel}_{p^\infty}(\Q_\infty, E)$ shall be detected by the structure of the residual representation. We shall first recall a result of Schneider on isogenies between elliptic curves.
\par Given a finite Galois stable submodule $\alpha$ of $E[p^\infty]$, set $\alpha^+:=C\cap \alpha$; set
\[\delta(\alpha):=\ord_p|\alpha^+|-\ord_p|H^0(\mathbb{R}, \alpha)|.\] We note that since $p$ is assumed to be odd, the above definition coincides with that of \cite[Definition 2.1]{drinen2002finite}. In particular, it is easy to see that the quantities $\epsilon_v=0$ from \emph{loc. cit.} are trivial.
\begin{theorem}[Schneider]\label{isogeny schneider}
    Let $E$ and $E'$ be elliptic curves with good ordinary reduction at $p$ and $\phi: E\rightarrow E'$ an isogeny with kernel $\alpha$. then the difference between $\mu$-invariants is given by 
    \[\mu_p(E)-\mu_p(E')=\delta(\alpha).\]
    In particular, it follows that $\mu_p(E)\geq \delta(\alpha)$.
\end{theorem}
\begin{proof}
    We refer to \cite{schneider1987mu} or \cite[Theorem 2.2]{drinen2002finite} for the proof of the above result.
\end{proof}

We recall a result of Coates and Sujatha which will be of key importance in the proof of Greenberg's conjecture in the residually reducible case. 

\begin{theorem}[Coates and Sujatha]\label{CS thm}
    Let $E$ be an elliptic curve over $\Q$ such that $\rho_{\bar{A}}$ is a reducible Galois representation. Then, the $\mu$-invariant of the fine Selmer group $\op{Sel}_{p^\infty}^0(\Q_\infty, E)$ is equal to $0$. 
\end{theorem}

\begin{proof}
    We write $\rho_{\bar{A}}=\mtx{\varphi_1}{0}{\ast}{\varphi_2}$ with respect to some basis of $\bar{A}$. Let $\mathcal{K}=\Q(\varphi_1, \varphi_2)$ be the abelian extension of $\Q$ generated by $\varphi_1$ and $\varphi_2$. Let $\mathcal{K}(E_{p^\infty})$ be the extension generated by the $p$-primary torsion points of $E$. In other words, $\mathcal{K}(E_{p^\infty})$ is the field extension of $\mathcal{K}$ which is fixed by the kernel of $\rho_M$. Let $I$ be the subgroup of $\op{GL}_2(\Z_p)$ consisting of all matrices $A$ for which the mod-$p$ reduction is a unipotent lower triangular matrix $\mtx{1}{0}{\ast}{1}$. Via $\rho_M:\op{G}_{\Q}\rightarrow \op{GL}_2(\Z_p)$, the Galois group $\op{Gal}(\mathcal{K}(E_{p^\infty})/\mathcal{K})$ is identified with a subgroup of $I$. Since $I$ is a pro-$p$ group, so is the Galois group $\op{Gal}(\mathcal{K}(E_{p^\infty})/\mathcal{K})$. Recall that by the celebrated result of Ferrero and Washington \cite{ferrero1979iwasawa}, the classical Iwasawa $\mu$-invariant $\mu_p(\mathcal{K})$ vanishes, since $\mathcal{K}$ is an abelian extension of $\Q$. It then follows from Theorem \ref{CS main result} that $\op{Sel}_{p^\infty}^0(\mathcal{K}_\infty, E)$ is cofinitely generated as a $\Z_p$-module. The kernel of the restriction map 
    \[\op{Sel}_{p^\infty}^0(\mathcal{\Q}_\infty, E)\rightarrow \op{Sel}_{p^\infty}^0(\mathcal{K}_\infty, E)\] is contained in $H^1(H, E(\mathcal{K}_\infty)[p^\infty])$, where $H=\op{Gal}(\mathcal{K}_\infty/\Q_\infty)$. Since $H$ has order prime to $p$, it follows that this cohomology group vanishes. Therefore, $\op{Sel}_{p^\infty}^0(\mathcal{\Q}_\infty, E)$ is cofinitely generated as a $\Z_p$-module. In particular, the $\mu$-invariant of $\op{Sel}_{p^\infty}^0(\Q_\infty, E)$ is $0$.
\end{proof}

\begin{theorem}\label{thm 6.2}
Let $E_{/\Q}$ be an elliptic curve with good ordinary reduction at an odd prime $p$ for which the following conditions hold.
\begin{enumerate}
    \item The residual representation is reducible. 
    \item The Assumption \ref{main assumption} holds.
    \item The Conjecture \ref{main conjecture} holds.
\end{enumerate}Then, we find that $\mu_p(E)=0$.
\end{theorem}

\begin{proof}
    The assumption \ref{main assumption} holds, and therefore, by part \eqref{p2 main 1} of Theorem \ref{main 1}, the $\mu$-invariant of $\op{Sel}_{p^\infty}(\Q_\infty, E)$ vanishes if and only if the $\mu$-invariant of $\op{Sel}_{p^\infty}^0(\Q_\infty, E)$ vanishes. Since the residual representation is reducible, it follows from Theorem \ref{CS thm} that the $\mu$-invariant of $\op{Sel}_{p^\infty}^0(\Q_\infty, E)$ is $0$, and the result follows from this. 
\end{proof}
Assuming that Conjecture \ref{main conjecture} holds for the isogeny class of $E$, we have a complete description for the $\mu=0$ condition based purely on the residual representation $\bar{\rho}_{E,p}$.
\begin{theorem}[$\mu=0$ condition]\label{mu=0 condition}\label{mu equals zero condition}
    Let $E_{/\Q}$ be an elliptic curve and $p$ an odd prime at which $E$ has good ordinary reduction. Assume that the Conjecture \ref{main conjecture} holds for all elliptic curves that are defined over $\Q$ and are $\Q$-isogenous to $E$. Then, $\mu_p(E)=0$ if and only if $\bar{\rho}_{E,p}$ is of type 2 or 3. Equivalently, $\mu_p(E)>0$ if and only if it is of type 1.
\end{theorem}
\begin{proof}
\par First, we consider the case when $\bar{\rho}_{E,p}$ is of type $1$. Note that since the representation $\bar{\rho}_{E, p}$ is upper triangular, $\alpha:=\bar{C}$ is a $\op{G}_{\Q}$-submodule of $E[p]$. Since $\varphi_1$ is odd, $H^0(\mathbb{R}, \alpha)=0$. Then, we find that $\delta(\alpha)=1$, and it follows from Theorem \ref{isogeny schneider} (or \cite[Theorem 2.1]{drinen2002finite}) that $\mu_p(E)\geq \delta(\alpha)\geq 1$. 

\par Next, we consider type $2$ representations. Greenberg and Vatsal \cite{greenberg2000iwasawa} showed that if $E[p]$ contains a $1$-dimensional $\op{G}_{\Q}$-stable subspace which is ramified at $p$ and even or unramified at $p$ and odd, then, $\mu_p(E)=0$. In this case, $\bar{C}$ is a subspace which is $\op{G}_{\Q}$-stable, ramified at $p$ and even, and therefore, their result applies to show that $\mu_p(E)=0$.

\par Finally, consider the type $3$ representations. Note that $\varphi_2$ is unramified at $p$. Thus, if $\varphi_2$ is odd, then the aforementioned result of Greenberg and Vatsal applies to show that $\mu_p(E)=0$. For type $3$ representations for which $\varphi_2$ is even however, it was expected that $\mu_p(E)=0$ should hold, however, not proved. We complete the proof by noting that Proposition \ref{assumption ell curve} implies that when $\bar{\rho}_{E, p}$ is of type $3$, the Assumption \ref{main assumption} holds. Then it follows from Theorem \ref{thm 6.2} that $\mu_p(E)=0$.
\end{proof}
We now give the proof of our main theorem.

\begin{proof}[Proof of Theorem \ref{main thm}]
If $\mu:=\mu_p(E)=0$, then the result is vacuously true, setting $E':=E$. Therefore, assume without loss of generality that $\mu>0$. Thus, it follows from Theorem \ref{mu=0 condition} that $\bar{\rho}_{E, p}$ is of type 1, i.e., $\bar{C}$ is an odd $\op{G}_{\Q}$-submodule of $E[p]$. In this case, setting $\alpha:=\bar{C}$, we observe that $\delta(\alpha)=1$ (see the first paragraph in the proof of Theorem \ref{mu=0 condition}). We set $E_1:=E/\alpha$. It follows from Theorem \ref{isogeny schneider} that
\[\mu_p(E_1)=\mu_p(E)-\delta(\alpha)=\mu-1.\]
In this way, we obtain a sequence of elliptic curves $E=E_0,E_1, E_2, \dots, E_{\mu}$ over $\Q$ along with isogenies $\phi_i: E_{i-1}\rightarrow E_i$ such that $\mu_p(E_i)=\mu-i$. Set $E':=E_{\mu}$ and consider the composite isogeny
\[E\xrightarrow{\phi_1}E_1\xrightarrow{\phi_2}E_2\xrightarrow{\phi_3}\dots \rightarrow E'.\]
We find that $\mu(E')=\mu-\mu=0$. This completes the proof.
\end{proof}

\bibliographystyle{alpha}
\bibliography{references}

\end{document}